\documentclass[11pt]{article}

\usepackage{graphicx}
\usepackage{epsfig}
\usepackage{epsf}
\usepackage{latexsym}
\usepackage{bbm}
\usepackage{float}
\usepackage{fancyvrb}
\usepackage[small]{caption}
\usepackage{url}
\usepackage{amsthm}
\usepackage{amsfonts}
\usepackage{amssymb}
\usepackage{amsmath}
\usepackage{enumerate}

\newtheorem{theorem}{Theorem}

\newtheorem{question}{Question}

\newcommand{\eps}{\varepsilon}

\newcommand{\NN}{\mathbb{N}} 
\newcommand{\QQ}{\mathbb{Q}} 
\newcommand{\RR}{\mathbb{R}} 

\def\J{\mathcal{J}}
\def\K{\mathcal{K}}

\newcommand{\len}{{\rm len}}

\newcommand{\ie}{{i.e.}}
\newcommand{\eg}{{e.g.}}

\newcommand{\later}[1]{{}}
\newcommand{\old}[1]{{}}
\long\def\ignore#1{}

\title{\textsc{Problems on Track Runners}\thanks{A preliminary
    version of this paper appeared in the {\em Proceedings of the 29th
      Canadian Conference on Computational Geometry}, Ottawa, ON,
    Canada, July 2017.
}}

\author{%
Adrian Dumitrescu\\
\small Department of Computer Science\\[-0.8ex]
\small University of Wisconsin--Milwaukee\\[-0.8ex]
\small Milwaukee, WI, USA\\
\small\tt dumitres@uwm.edu\\
\and
Csaba D. T\'oth\\
\small Department of Mathematics\\[-0.8ex]
\small California State University, Northridge\\[-0.8ex]
\small Los Angeles, CA, USA\\
\small\tt cdtoth@acm.org
}

\begin{document}

\maketitle

\begin{abstract}
  Consider the circle $C$ of length 1 and a circular arc $A$ of length $\ell\in (0,1)$.
  It is shown that there exists $k=k(\ell) \in \NN$, and a schedule for $k$
  runners along the circle with $k$ constant but distinct positive speeds so that at
  any time $t \geq 0$, at least one of the $k$ runners is \emph{not} in $A$.

  On the other hand, we show the following.
  Assume that $k$ runners $1,2,\ldots,k$, with constant rationally independent
  (thus distinct) speeds $\xi_1,\xi_2,\ldots,\xi_k$,
  run clockwise along a circle of length $1$, starting from
  arbitrary points. For every circular arc $A\subset C$
  and for every $T>0$, there exists $t>T$  such that
  all runners are in $A$ at time $t$.

  Several other problems of a similar nature are investigated.

\end{abstract}

\section{Introduction} \label{sec:intro}

In the classic \emph{lonely runner conjecture},
introduced by Wills~\cite{Wi67} and Cusick~\cite{Cu73},
$k$ runners run clockwise along a circle of length $1$,
starting from the same point at time $t=0$.
They have distinct but constant speeds. A runner is called \emph{lonely} when he/she is
at distance of at least $\frac{1}{k}$ from all other runners (along the circle).
The conjecture asserts that each runner $i$ is lonely at some time
$t_i\in (0,\infty)$. The conjecture has only been confirmed for up to
$k=7$ runners~\cite{BS08,BHK01}. A recent survey~\cite{DT14} lists
a few other related problems.

Recently, some problems with similar flavor have appeared in the context of
\emph{multi-agent patrolling}, in some one-dimensional
scenarios~\cite{Ch04,CGKK11,DGT14,KK15,KS15}.
Suppose that $k$ mobile agents with (possibly distinct) maximum speeds
$v_i$ ($i=1,\ldots,k$) are in charge of \emph{patrolling} a closed or open
fence (modeled by a circle or a line segment). The movement of the agents
over the time interval $[0,\infty)$ is described by a \emph{patrolling schedule}
(or \emph{guarding schedule}), where the speed of the $i$th agent, ($i=1,\ldots,k$),
may vary between zero and its maximum value $v_i$ in any of the two directions
along the fence. Given a closed or open fence of length $\ell$ and maximum speeds
$v_1,\ldots,v_k>0$ of $k$ agents, the goal is to find a patrolling schedule that minimizes
the \emph{idle time}, defined as the longest time interval in $[0,\infty)$
during which some point along the fence remains unvisited, taken over all points.
Several basic problems are open, such as the following:
Given $v_1,\ldots,v_k>0$ and $\ell,\tau>0$,
can one even decide whether $k$ agents with these maximum speeds
can ensure an idle time at most $\tau$ when patrolling a segment of length $\ell$?

This paper is devoted to several questions on track runners.
As customary, we consider the unidirectional circular track;
for convenience we assume all runners run clockwise.
In the spirit of the lonely runner conjecture, we posed the following
question in~\cite{DT14} (slightly rephrased here):

\begin{question}\label{q:original}
Assume that $k$ runners $1,2,\ldots,k$, with constant but distinct speeds,
run clockwise along a circle of length $1$, starting from arbitrary points.
Assume also that a certain half of the circular track (or any other fixed
circular arc $A$) is in the shade at all times. Does there always exist a time
when all runners are in the shade along the track?
\end{question}

Here we answer Question~\ref{q:original} in the negative: the statement does not hold
even if the shaded arc almost covers the entire track,
\eg, has length $0.999$, provided that $k$ is large enough.

\paragraph{Notation and terminology.}
We parameterize a circle of length $\ell$ by the interval $[0,\ell]$, where
the endpoints of the interval $[0,\ell]$ are identified.
A \emph{unit} circle is a circle of unit length $C=[0,1) \mod 1$.
A \emph{schedule} of $k$ agents consists of $k$ functions
$f_i:[0,\infty] \rightarrow [0,\ell]$, for
$i=1,\ldots , k$, where $f_i(t) \mod \ell$ is the position of agent
$i$ at time $t$. Each function $f_i$ is continuous,
piecewise differentiable, and its derivative (speed) is bounded by $|f_i'|\leq v_i$.
The $k$ agents have \emph{constant speeds} $v_1,\ldots, v_k$,
with starting points $\beta_1,\ldots, \beta_k$
when $f_i(t)=v_i t+\beta_i \mod \ell$ for all $i=1,\ldots , k$.
A schedule is called \emph{periodic} with {period} $T>0$
if $f_i(t)=f_i(t+T) \mod \ell$ for all $i=1,\ldots , k$ and $t\geq 0$.
$H_n= \sum_{i=1}^n 1/i$ denotes the $n$th \emph{harmonic number}; and $H_0=0$.
If $I$ is an interval, $|I|$ denotes its length.

If $I=[a,b]$ is an interval, $x \in \RR$, and $y>0$, then $I+x$ is the interval $[a+x,b+x]$
and $yI$ is the interval $[ay,by]$; this notation is used in Section~\ref{sec:algo}.

\section{Track runners in the shade} \label{sec:shade}

We first show that the answer to the question posed in~\cite{DT14} is negative in general:

\begin{theorem}\label{thm:no-shade}
  Consider a circle $C$ of unit length and a circular arc $A\subset C$ of length
  $\ell=|A|$, where $\ell\in (0,1)$. Then there exists $k=k(\ell) \in \NN$, and a schedule for $k$
  runners with $k$ distinct constant speeds and suitable starting points,
  so that at any time $t \geq 0$, at least one of the $k$ runners is in
  the complement $C \setminus A$.
\end{theorem}
\begin{proof}
Set $v_i=i$ as the speed of runner $i$, for $i=1,\ldots,k$, where $k=k(\ell)\in \NN$
will be specified later. Assume, as we may, that $C \setminus A=[0,a]$, where $a=1-\ell$.
Let $t_0=0$. Since the speed of each runner is an integer
(and thereby multiple of the circle length $\len(C)=1$),
the resulting schedule is periodic and the period is $1$.
To ensure that at any time $t \geq 0$, at least one runner is in $[0,a]$,
it suffices to ensure this
\emph{covering condition} on the time interval $[0,1)$, \ie, one period of the schedule.
All runners start at time $t=0$; however, it is convenient to specify their schedule
with their positions at a later time.

Runner $1$ starts at point $0$ at time $0$; at time $a$, its position is at $a$ (exiting $[0,a]$).
Runner $2$ is at point $0$ at time $a$; at time $a+a/2$, its position is at $a$  (exiting $[0,a]$).
Runner $3$ is at point $0$ at time $a+a/2$; at time $a+a/2+a/3$, its position is at $a$
(exiting $[0,a]$). Subsequent runners are scheduled according to this pattern.
For $i=1,\ldots,k$, runner $i$ is at point $0$ at time $a H_{i-1}$;
at time $a H_i$, its position is at $a$ (exiting $[0,a]$).
The schedules are given by the functions $f_i(t) = i t - i a H_{i-1}$ for $i=1,\ldots,k$.

The construction ensures that runner $i$ is in $[0,a]$ during the time interval
$[aH_{i-1},aH_i]$, for $i=1,\ldots,k$. We choose $k\in \mathbb{N}$ such that
\begin{equation}\label{eq:hn}
\bigcup_{i=1}^k [aH_{i-1},aH_i] \supseteq [0,1),
\end{equation}
and then at least one of the $k$ runners will be in $[0,a]$
at any time $t\geq 0$, as required.
Condition \eqref{eq:hn} is equivalent to $a H_k \geq 1$, or $H_k \geq 1/a$.
Since $\ln{k} \leq H_k$, it suffices to have $\ln{k} \geq  1/a$, or
$k \geq \exp(1/a)$, and the theorem is proved.
\end{proof}

The result extends to any finite number of circular arcs $A_1,A_2,\ldots ,A_m\subset C$.
Stating the results for the complements $B_i=C \setminus A_i$, for $i=1,2,\ldots, m$,
we can schedule $k$ mobile runners with distinct integer speeds so that at any time $t\geq 0$,
each interval $B_i$ contains at least one of the runners.

\begin{theorem}\label{thm:many-no-shade}
  Consider a circle $C$ of unit length and $m$ circular arcs $B_1,B_2,\ldots , B_m\subset C$,
  for some $m \in \NN$.
  Then there exists $k \in \NN$, and a schedule for $k$ runners with $k$ distinct
  constant speeds and suitable starting points, so that at any time $t \geq 0$,
  each of the arcs $B_1,B_2,\ldots, B_m$ contains at least one of the $k$ runners.
\end{theorem}
\begin{proof}
In the proof of Theorem~\ref{thm:no-shade}, we constructed a schedule of $k(\ell)$ runners
with speeds $1,2,\ldots, k(\ell)$. Note, however, that for any $s \in \NN$,
we could have used runners of speeds $s+1,s+2,\ldots , s+k(\ell,s)$, such that
\begin{equation}\label{eq:s}
\bigcup_{i=s+1}^{s+k(\ell,s)} [aH_{i-1},aH_i] \supseteq [0,1).
\end{equation}
Indeed, for every $s \in \NN$ there exists $k(\ell,s) \in \NN$ satisfying~\eqref{eq:s},
since $\lim_{i\rightarrow \infty} H_i=\infty$.

For arc $B_1=[0,a_1]$, there exists $k_1=k(|B_1|,0)\in \mathbb{N}$ and a schedule
for $k_1$ runners with speeds $1,2,\ldots, k_1$ such that at least one of these runners
is in $B_1$ at any time $t\geq 0$.
For arc $B_2$, there exists $k_2=k(|B_2|,k_1)\in \mathbb{N}$ and a schedule for $k_2$
runners with speeds $k_1+1,k_2+2,\ldots, k_1+k_2$ such that at least one of them is in $B_2$
at any time $t\geq 0$.
In general, if the first $i-1$ intervals are covered,
let $s_i=\sum_{j=1}^{i-1}k_i$. Then there exists $k_i=k(|B_i|,s_i)\in \mathbb{N}$ and
a schedule for $k_i$ runners with speeds $s_i+1,s_i+2,\ldots, s_i+k_i$ such that
at least one of them is in arc $B_i$ at any time $t\geq 0$, as required.
\end{proof}

\paragraph{Positive cases.}
Now that we have seen that the answer to Question~\ref{q:original} is negative in general,
it is however interesting to exhibit some scenarios (\ie, conditions)
under which the answer is positive.

A set of real numbers $\xi_1,\xi_2,\ldots,\xi_k$ is said to be \emph{rationally independent}
if no linear relation
$$ c_1 \xi_1 + c_2 \xi_2 + \cdots +  c_k \xi_k = 0, $$
with integer coefficients, not all of which are zero, holds.
In particular, if $\xi_1,\xi_2,\ldots,\xi_k$ are rationally independent,
then they are pairwise distinct.
Recall now Kronecker's theorem; see, \eg, \cite[Theorem 444, p.~382]{HW79}.
\begin{theorem} {\rm (Kronecker, 1884)} \label{thm:K}
If $\xi_1,\xi_2,\ldots,\xi_k\in \mathbb{R}$ are rationally independent,
$\alpha_1,\alpha_2,\ldots,\alpha_k\in \mathbb{R}$ are arbitrary,
and $T$ and $\eps$ are positive reals, then there is a real number $t>T$,
and integers $p_1,p_2,\ldots,p_k$, such that
$$ | t \xi_m -p_m - \alpha_m | < \eps \ \ \ (m=1,2,\ldots,k). $$
\end{theorem}

As a corollary, we obtain the following result.

\begin{theorem} \label{thm:shade1}
  Assume that $k$ runners $1,2,\ldots,k$, with constant rationally independent
  (thus distinct) speeds $\xi_1,\xi_2,\ldots,\xi_k$,
  run clockwise along a circle of length $1$, starting from
  arbitrary points. For every circular arc $A\subset C$
  and for every $T>0$, there exists $t>T$  such that
  all runners are in $A$ at time $t$.
\end{theorem}
\begin{proof}
  Assume, as we may, that $A=[0,a]$, for some $a \in (0,1)$.
  Let $0 \leq \beta_i < 1$, be the start position of runner $i$, for $i=1,2,\ldots,k$.
  Set $\alpha_i = a/2 + 1- \beta_i$, for $i=1,2,\ldots,k$,
  set $\eps=a/3$, and employ Theorem~\ref{thm:K} to finish the proof.
\end{proof}

\paragraph{Remark.}
It is interesting to note that Theorem~\ref{thm:no-shade} gives a negative answer
to Question~\ref{q:original} regardless of how long the shaded arc is,
while Theorem~\ref{thm:shade1} gives a positive answer regardless
of how short the shaded arc is \emph{and} for how far in the future one desires.

Observe that if $\xi_1,\xi_2,\ldots,\xi_k$ are rationally independent reals,
then at least one $\xi_i$ must be irrational
(in fact, all but at most one $\xi_i$ must be irrational).
To obtain the conclusion of Theorem~\ref{thm:shade1} neither the condition that the
speeds $\xi_1,\xi_2,\ldots,\xi_k$ are rationally independent, nor the condition that at least one
$\xi_i$ is irrational is necessary.
We next show that one can incrementally choose rational speeds for the $k$ runners
such that all are in the shade infinitely many times (regardless of their starting points).

\begin{theorem} \label{thm:shade2}
For every circular arc $A\subset C$, $k\in \mathbb{N}$, and starting positions
$\beta_1,\beta_2,\ldots , \beta_k\in C$, there exist distinct rational speeds
$v_1,v_2,\ldots,v_k>0$, such that the following holds.
If $k$ runners run clockwise with constant speeds $v_1,v_2,\ldots , v_k$
starting from $\beta_1,\beta_2,\ldots , \beta_k$, then for every $T>0$,
there exists $t>T$ such that all runners are in $A$ at time $t$.
\end{theorem}
\begin{proof}
  Assume, as we may, that $[0,a]\subseteq A$, for some rational $a \in (0,1)$.
  Let $\beta_1,\beta_2,\ldots,\beta_k$ be the starting points of the runners,
  where $0 \leq \beta_i < 1$, for $i=1,2,\ldots,k$.
  We proceed by induction on the number of runners $k$, and
  with a stronger induction hypothesis extending to every arc $A$.
  The base case $k=1$ is satisfied by setting $v_1=1$ for any arc $A$.
  The subsequent speeds will be set to increasing values, so that
  $v_1 < v_2 < \cdots <v_k$.

 For the induction step, assume that the statement holds for runners $1,2,\ldots,k-1$,
 the arc $A'=[0,a/2]$ and $T$, and we need to prove it for runners $1,2,\ldots,k$,
 the arc $A=[0,a]$ and $T$.
 By the induction hypothesis, there exists $t>T$ so that
 runners $1,2,\ldots,k-1$, are in $A'$ at time $t$.
 Set $v_k =\frac{2}{a} v_{k-1}$; since $a,v_{k-1} \in \QQ$, we have $v_k \in \QQ$.
 Observe that runner $k$ will enter the arc $A$ at point $0$ before any of the first $k-1$
 runners exits $A$ at point $a$, regardless of his or her starting point.
 Hence all $k$ runners will be in $A$ at some time in the
 interval $[t,t+1/v_k]$, completing the induction step, and thereby the proof of the theorem.
\end{proof}

In Theorem~\ref{thm:shade2}, the speeds $v_1,v_2,\ldots , v_k$
ensure that $k$ runners are in a given circular arc $A\subset C$ infinitely
many times. Different intervals may require different speeds (based on the
relative position of $A$ and the $k$ starting positions). The next theorem
shows that, in fact, the same $k$ speeds ensure this property for all
circular arcs of a given length $a>0$. Its proof is very similar to that
of Theorem~\ref{thm:shade2}; for clarity we include both proofs.

\begin{theorem} \label{thm:shade3}
For every $a\in (0,1]$, $k\in \mathbb{N}$, and starting positions
$\beta_1,\beta_2,\ldots , \beta_k\in C$, there exist distinct rational speeds
$v_1,v_2,\ldots,v_k>0$, such that the following holds.
If $k$ runners run clockwise with constant speeds $v_1,v_2,\ldots , v_k$
starting from $\beta_1,\beta_2,\ldots , \beta_k$, then for every $T>0$
and every circular arc $A\subset C$ of length $a$,
there exists $t>T$ such that all runners are in $A$ at time $t$.
\end{theorem}
\begin{proof}
  Let $\beta_1,\beta_2,\ldots,\beta_k$ be the starting points of the runners,
  where $0 \leq \beta_i < 1$, for $i=1,2,\ldots,k$.
  We proceed by induction on the number of runners $k$.
  The base case $k=1$ is satisfied by setting $v_1=1$ for any $a>0$.

 For the induction step, assume that the statement holds for $k-1$ runners,
 and we need to prove it for $k$ runners. Let an arc length length $a>0$ and
 $k$ starting positions $\beta_1,\ldots , \beta_k$ be given. By the induction
 hypothesis, for the arc length $a'=a/2$ and $k-1$ starting points $\beta_1,\ldots, \beta_{k-1}$
 there exist speeds $v_1,\ldots v_{k-1}$ so that for any $T\geq 0$ and any arc $A'\subset C$
 of length $a'=a/2$, all runners $1,2,\ldots,k-1$ are in $A'$ at some time $t>T$.

 Set $v_k =\frac{2}{a} v_{k-1}$. Consider an arbitrary arc
 $A=[\alpha,\alpha+a]\subset C$ of length $a$. Denote the first half of
 the arc by $A'=[\alpha,\alpha+a/2]$. At time $t$, runners $1,2,\ldots,k-1$ are in $A'$.
 Observe that runner $k$ will enter the arc $A$ at point $\alpha$ before any of the first $k-1$
 runners exits $A$ at point $\alpha+a$, regardless of his or her starting point.
 Hence all $k$ runners will be in $A$ at some time in the interval $[t,t+1/v_k]$,
 completing the induction step, and thereby the proof of the theorem.
\end{proof}

\noindent
The speeds of the runners in Theorems~\ref{thm:shade2} and~\ref{thm:shade3} can be chosen
as integers if desired, by setting $v_k =\lceil \frac{2}{a} v_{k-1}\rceil$.

\section{Algorithmic aspects and concluding remarks} \label{sec:algo}

It would be interesting to know if Theorem~\ref{thm:no-shade} can be strengthened
for runners starting at the same point. Question~\ref{q:original} then becomes:

\begin{question} \label{q:new}
Assume that $k$ runners $1,2,\ldots,k$, with constant but distinct speeds,
run clockwise along a circle of length $1$, starting at $0$.
Assume also that a certain circular arc is in the shade at all times. Does there always
exist a time when all runners are in the shade along the track?
\end{question}

It is worth pointing out a connection between runners in the shade and idle time
(as defined in Section~\ref{sec:intro}). Assume that $k$ runners $1,2,\ldots,k$, with
constant rationally independent (thus distinct) speeds $0< \xi_1,\xi_2,\ldots,\xi_k \leq 1$,
run clockwise along a circle of length $1$, starting from arbitrary points.
Further assume that $\sum_{i=1}^k \xi_i =S$, where $S \leq k$ is large, say, close to $k$.
A straightforward volume argument~\cite{CGKK11} yields the lower bound
$\tau \geq 1/\sum_{i=1}^{k} \xi_i =1/S \geq 1/k$ on the idle time. On the other hand,
by Theorem~\ref{thm:K}, for every circular arc $A\subset C$
and for every $T>0$, there exists $t>T$  such that
all runners are in $A$ at time $t$; pick an arbitrary interval $A$ of length $|A|=\eps$,
where $\eps$ is small. Since the maximum speed of the runners is at most $1$, the idle
time $\tau$ must be at least $|C \setminus A|=1-\eps$. The example shows that
the volume-based lower bound for the idle time can be very weak for large~$k$.

In view of this connection with idle time, it is unclear whether the following
basic question is decidable:

\begin{question} \label{prob:idle}
Given $v_1,\ldots,v_k>0$ and $\ell,\tau>0$, is there a schedule for
$k$ agents with these maximum speeds that ensures an idle time at most $\tau$
when patrolling a segment of length $\ell$?
\end{question}

\smallskip
The questions we have studied also suggest a few algorithmic questions for a circle $C$
of unit length that we list below.
Even if all runners start from the same point (say, $\beta_i=0$ for all $i=1,2,\ldots,k$),
it is a~priori unclear how to test whether some runner will be in the shade at all times
(or from some time later on), or all runners will be out of the shade infinitely often.
(In Questions~\ref{q:at-least1},~\ref{q:at-least2},~\ref{q:inverse1},~\ref{q:inverse2},
the circular arc $A$ is the complement of the shaded part.)

\begin{question} \label{q:all1}
Given $k$ runners with speeds $v_1,\ldots, v_k>0$ starting at $0$
and a circular arc $A \subset C$, decide whether there exists $t \geq 0$,
such that all $k$ runners are in $A$ at time $t$.
\end{question}

\begin{question} \label{q:all2}
Given $k$ runners with speeds $v_1,\ldots, v_k>0$ starting at $0$,
and an circular arc $A \subset C$,
decide whether for every $T \geq 0$ there exists $t \geq T$ such that
all $k$ runners are in $A$ at time $t$.
\end{question}

\begin{question} \label{q:at-least1}
Given $k$ runners with speeds $v_1,\ldots, v_k>0$ starting at $0$
and a circular arc $A \subset C$, decide whether there exists $T \geq 0$,
such that at every time $t\geq T$, at least one of the runners is in~$A$.
\end{question}

\begin{question} \label{q:at-least2}
Given $k$ runners with speeds $v_1,\ldots, v_k>0$ and a circular arc $A \subset C$,
decide whether there exist starting points $\beta_1,\ldots,\beta_k$ for the $k$ runners,
such that at every time $t\geq 0$, at least one of the runners is in~$A$.
\end{question}

The following two questions are in some sense the ``inverses'' of Question~\ref{q:original}.

\begin{question} \label{q:inverse1}
Given $k$ runners with speeds $v_1,\ldots, v_k>0$ starting from points
$\beta_1,\ldots , \beta_k\in C$, respectively, and an arc length $\ell>0$,
decide whether there exist a circular arc $A \subset C$ of length $\ell$
and a time $T\geq 0$ such that at every time $t\geq T$
at least one of the runners is in~$A$.
\end{question}

\begin{question} \label{q:inverse2}
Given $k$ runners with starting points $\beta_1,\ldots,\beta_k \in C$,
a circular arc $A \subset C$, and a parameter $v>0$,
decide whether there exist rational speeds $v_1,\ldots,v_k \in (0,v)$ and a time $T\geq 0$
such that at every time $t\geq T$ at least one of the runners is in~$A$.
\end{question}

\paragraph{Algorithms for periodic schedules.}
The challenge in Questions~\ref{q:all1}--\ref{q:inverse2} lies in dealing with
\emph{irrational} speeds. If all speeds are rational and fixed (not
just upper bounds), then the schedules of the runners are periodic,
where the period is the common denominator of the rational speeds.
Conversely, if a schedule with fixed speeds is periodic with period $q\in \mathbb{N}$
on the circle of unit length, then every speed can be written as
$v_i=\frac{r_i}{q}$ for some $r_i \in \NN$.

If we are given the speeds and starting points of $k$ agents, and the schedules are
periodic with period $q$, then we can compute the arrangement of the trajectories
of the agents in a time-position diagram $[0,q) \times [0,1)$.
Recall that the schedule of agent $i$ within a period $q$ is a function
    $f_i:[0,q) \rightarrow [0,1)$; its
    \emph{trajectory} is the graph $\{(t,f_i(t)): t\in [0,q)\}$.
If the speed of agent $i$ is $v_i=\frac{r_i}{q}$,
then its trajectory consists of at most $r_i+1$ line segments in this diagram.
The $k$ trajectories form an arrangement of at most $K=\sum_{i=1}^k (r_i+1)$ line segments,
and the complexity of the arrangement (the total number of vertices, edges, and faces)
is $O(K^2)$.
Questions~\ref{q:all1}--\ref{q:at-least1} and~\ref{q:inverse1},
in particular, can be answered based on this arrangement in time polynomial in $K$;
consequently these questions admit pseudopolynomial algorithms.
(Such an algorithm runs in time bounded by a polynomial in the unary complexity of the input;
see, \eg, \cite[p.~59]{WS11} for technical terms.)
As an example, we present algorithmic solutions for Questions~\ref{q:all1} and~\ref{q:all2}
(Theorem~\ref{thm:alg1}) and Question~\ref{q:inverse1} (Theorem~\ref{thm:alg2}).

\begin{theorem}\label{thm:alg1}
Given $k$ runners with rational (constant) distinct speeds $v_1,\ldots,v_k>0$
starting at $0$ and a circular arc $A \subset C$,
there exists an algorithm for deciding whether there exists $t \geq 0$,
such that all $k$ runners are in $A$ at time $t$.

Further, the same algorithm can decide whether for any $T \geq 0$
there exists $t \geq T$ such that all $k$ runners are in $A$ at time $t$.

Assuming that $v_i=\frac{r_i}{q}$, where $q \in \NN$, and $r_i \in \NN$ for $i=1,\ldots ,k$,
and $1 \leq r_1 < \ldots < r_k$, the algorithm runs in time $O(\sum_{i=1}^k i \, r_i) = O(k^2 \, r_k)$.
In particular, the running time is polynomial when $r_k$ is (bounded by a) polynomial in $k$.
\end{theorem}
\begin{proof}
Assume first, for the simplicity of exposition, that the speeds are integers:
$1 \leq v_1<\cdots<v_k$, where $v_i \in \NN$.
Since the speed of each runner is a multiple of the circle length, $\len(C)=1$,
the resulting schedule is periodic with period $1$
(as in the proof of Theorem~\ref{thm:no-shade}).
As such, it suffices to analyze what happens in the time interval $[0,1)$;
moreover, the answers to Questions~\ref{q:all1} and~\ref{q:all2} are the same.
Assume, as we may, that $A=[a,b]$, where $0<a<b<1$ and $a,b \in \QQ$
(since otherwise, $t=0$ and $t\in \NN$ are trivial solutions for
Questions~\ref{q:all1} and~\ref{q:all2}, respectively).

For $i=1,\ldots,k$, let $\J_i$ denote the set of time intervals in $[0,1)$
during which runner $i$ is in $I$; and let
$\K_i$ denote the set of time intervals in $[0,1)$ during which
runners $1$ through $i$ are \emph{all} in $I$.
The algorithm iteratively computes $\J_i$ and $\K_i$, for $i=1,\ldots,k$.
Observe that $\J_i$ consists of $v_i$ intervals of equal length: $|\J_i| = v_i$.
Since the start point of each interval in $\K_i$ can be uniquely associated with
the start point of one of the intervals in $\K_j$, for some $j \leq i$,
this implies that  $|\K_i| \leq \sum_{j=1}^i v_j$ for $i=1,\ldots,k$
(with equality for $i=1$); it is worth noting that the sharper inequality
$|\K_i| \leq v_i$ may not hold.

For Question~\ref{q:all1}, the algorithm outputs YES and returns some interval
(or all intervals) in $\K_k$, if $\K_k \neq \emptyset$, and NO otherwise.

We have $\J_1=\K_1$.
For every $i$, the intervals in $\J_i$ are sequentially computed from left to right:
$$ \J_i =\left\{ \frac{1}{v_i} \left(h+I\right) : h=0,\ldots,v_i-1\right\}. $$
For $i=2,\ldots,k$, the intervals in $\K_i$ are sequentially computed from left to right
in a merge-like process taking the intersection between the current interval in $\J_i$
and the current interval in $\K_{i-1}$.
The running time in step $i$ is $O(\sum_{j=1}^i v_j) = O(i \, v_i)$; consequently,
the overall running time is
$O(\sum_{i=1}^k \sum_{j=1}^i v_j ) = O(\sum_{i=1}^k i \, v_i)= O(k^2 \, v_k)$, as required.

Consider now the general case with rational speeds $v_i=\frac{r_i}{q}$, where
$q \in \NN$ and $1 \leq r_1 < \cdots < r_k$ are natural numbers. Let $q$ be
the minimal denominator with this property, and so $\gcd(r_i,q)=1$ for at least one index $i$.
Since each speed is a multiple of $1/q$, the resulting schedule is periodic with period $q$;
and so it suffices to analyze what happens in the time-interval $[0,q)$.
We have $\J_1=\K_1$.
For every $i$, the intervals in $\J_i$ are sequentially computed from left to right:
$$ \J_i =\left\{ \frac{q}{r_i} \left(h+I\right) : h=0,\ldots,r_i -1 \right\}. $$
For $i=2,\ldots,k$, the intervals in $\K_i$ are computed in a merge-like process
(similarly to the case of integral speeds).
We have $|\J_i| = r_i$ and  $|\K_i| \leq \sum_{j=1}^i r_j$ for $i=1,\ldots,k$
(with equality for $i=1$).
The running time in step $i$ is $O(\sum_{j=1}^i r_j) = O(i \, r_i)$; consequently,
the overall running time is
$O(\sum_{i=1}^k \sum_{j=1}^i r_j ) = O(\sum_{i=1}^k i \, r_i)= O(k^2 \, r_k)$, as claimed.
\end{proof}

\begin{theorem}\label{thm:alg2}
Given $k$ runners with rational (constant) speeds $v_1,\ldots, v_k>0$ starting
from points $\beta_1,\ldots , \beta_k\in C$, respectively, and an arc length $\ell>0$,
there exists an algorithm for deciding whether there exist a circular arc
$A \subset C$ of length $\ell$ and a time $T\geq 0$ such that at every time
$t\geq T$ at least one of the runners is in $A$.

Assuming that $v_i=\frac{r_i}{q}$, where $q \in \NN$, and $r_i \in \NN$ for $i=1,\ldots,n$,
and $1 \leq r_1 < \ldots < r_k$, the algorithm runs in time $O(k^2 \, r_k^2)$.
In particular, the running time is polynomial when $r_k$ is (bounded by a) polynomial in $k$.
\end{theorem}
\begin{proof}
As noted above, the schedule is periodic with period $q$, and the trajectory of agent $i$
consists of at most $r_i+1$ line segments in the time-position diagram $[0,q) \times [0,1)$.
The first segment starts from $(0,\beta_i)$, the last segment ends at $(q,\beta_i)$,
and all other start and endpoints of the segments are on the vertical lines $t=0$ and $t=q$,
respectively. For a position $x \in [0,1)$, let $A(x)\subset C$ (resp., $B(x)\subset C$)
denote a circular arc of length $\ell$ whose clockwise first (resp., last) endpoint is at $x$.
Runner $i$ is in arc $A(x)$ at time $t$ if and only if $x \in B(f_i(t))$.
For runner $i$, denote the set of pairs $(t,x)$ such that $x \in B(f_i(t))$ by
$D_i=\{ (t,x): x\in B(f_i(t))\}$. 
Note that $D_i$ consists of at most $r_i$ connected components: it lies between the 
graphs of the functions $f(t)$ and $f_i(t)-\ell\mbox{ \rm mod } 1$
for $t\in [0,q)$, that is, $D_i$ is bounded by $O(r_i)$ line segments;
see Fig.~\ref{fig:diagram} for an example.

\begin{figure}[htbp]
\begin{center}
\includegraphics[width=0.9\textwidth]{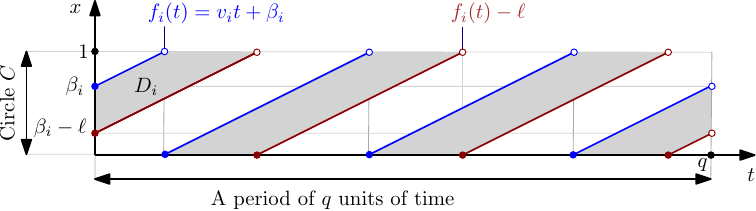}
\end{center}
\caption{The trajectory of agent $i$ over the interval $[0,q)$, where $f_i(t)=v_it+\beta_i$ 
and $v_i=r_i/q$. The shaded region $D_i$ is bounded by the trajectory of agent $i$,
and its vertical translate by $\ell$.}
\label{fig:diagram}
\end{figure}

Define the union $D=\bigcup_{i=1}^k D_i$. A desired circular arc $A\subset C$
exists if and only if $D$ contains the horizontal segment $[0,q) \times \{x\}$
for some $x\in [0,1)$. Since the arrangement of the boundaries of the regions $D_1,\ldots , D_k$
consists of at most $K :=\sum_{i=1}^k(r_i+1) = O(kr_k)$ segments, the full arrangement as well as
$D$ can be computed in $O(K^2) = O(k^2r_k^2)$ time~\cite[Section~8.3]{Dutch}.
By sweeping the arrangement by a horizontal line, it can be determined in $O(k^2r_k^2)$ time
whether $D$ contains a horizontal segment $[0,q] \times \{x\}$.
\end{proof}

Designing algorithms for Questions~\ref{q:all1}--\ref{q:inverse2} that run
in polynomial time, \ie, in time polynomial in the bit complexity of the
input, $O(\log{q} + \sum_{i=1}^k \log{r_i})$, is left open.
The existence of polynomial-time algorithms in the real RAM model
(for arbitrary real input) for Questions~\ref{q:all1}--\ref{q:inverse2}
remains open, as well.

\end{document}